\documentclass[a4paper,reqno]{amsart}

\usepackage{amsmath,amsfonts,amsthm,amssymb,latexsym,mathrsfs,enumerate}

\usepackage{hyperref}
\usepackage[all]{xy}

\newcommand{\diag}{\mathrm{diag}}

\newcommand{\ped}{\mathrm{Ped}}

\newcommand{\cu}{\mathrm{Cu}}
\newcommand{\cs}{\mathrm{C}^*}
\newcommand{\aff}{\mathrm{Aff}_+}

\newcommand{\saff}{\mathrm{SAff}_+}

\newcommand{\nn}{\mathbb{N}}

\newcommand{\cc}{\mathbb{C}}

\newenvironment{items}{\begin{list} {$\cdot$} {\setlength{\leftmargin}{0.5cm}}}{\end{list}}

\newenvironment{sub}[1]{
\smallskip\noindent
\textsc{#1.}
}

\newtheoremstyle{smallcaps}
    {3pt}                    
    {3pt}                    
    {\itshape}                   
    {}                           
    {\sc}                   
    {.}                          
    {.5em}                       
    {}  
    
\newtheoremstyle{smallcapsdef}
    {3pt}                    
    {3pt}                    
    {}                   
    {}                           
    {\sc}                   
    {.}                          
    {.5em}                       
    {}  

\theoremstyle{smallcaps}

\newtheorem {thm}{Theorem}[section]
\newtheorem {lemma}[thm]{Lemma}
\newtheorem {prop}[thm]{Proposition}
\newtheorem {corollary}[thm]{Corollary}

\theoremstyle {smallcapsdef}

\newtheorem {remark}[thm]{Remark}

\numberwithin{equation}{section}

\begin{document}

\author{Bhishan Jacelon}
\date{\today}
\title{$\mathcal{Z}$-stability, finite dimensional tracial boundaries and continuous rank functions}

\address{Mathematisches Institut\\Einsteinstrasse 62\\48149 M\"unster, Germany} 
\email{b.jacelon@uni-muenster.de}
\thanks{Research supported by the ERC through AdG 267079.}

\begin{abstract}
We observe that a recent theorem of Sato, Toms--White--Winter and Kirchberg--R{\o}rdam also holds for certain nonunital $\cs$-algebras. Namely, we show that an algebraically simple, separable, nuclear, nonelementary $\cs$-algebra with strict comparison, whose cone of densely finite traces has as a base a Choquet simplex with compact, finite dimensional extreme boundary, and which admits a continuous rank function, tensorially absorbs the Jiang--Su algebra $\mathcal{Z}$.
\end{abstract}
\maketitle

\section{Introduction}

As part of the project of investigating regularity properties of (simple, separable, unital, nonelementary, nuclear) $\cs$-algebras, there has been recent progress in showing that, in certain contexts, tensorial absorption of the Jiang--Su algebra $\mathcal{Z}$ (`$\mathcal{Z}$-stability') is an automatic consequence of strict comparison of positive elements. This was proved by Matui and Sato in \cite{Matui:2012qv} in the case of finitely many extremal tracial states, and later extended independently by Sato \cite{Sato:2012fj}, Toms, White and Winter \cite{Toms:2012qy} and Kirchberg and R\o rdam \cite{Kirchberg:2012uq} to include $\cs$-algebras with compact, finite-dimensional extreme tracial boundary. In this note, we point out that this more general result also holds for algebraically simple $\cs$algebras with compact tracial-state-space (and hence for those which admit a continuous rank function, in the sense of \cite{Dadarlat:2010nr}). We make use of Nawata's extension \cite{Nawata:2012fk} of Matui and Sato's techniques to the nonunital setting, of the language of `uniformly tracially large' order zero maps developed in \cite{Toms:2012qy}, and also of the philosophy of \cite{Tikuisis:2012kx}, namely that algebraically simple $\cs$-algebras can often be studied using the machinery developed for their unital cousins.

After establishing notation and tools in Section~\ref{prelim}, we present the main theorem in Section~\ref{mainsec} (following the exposition of \cite{Toms:2012qy} and  \cite{Sato:2012fj}). We then show in Section~\ref{further} that for $\cs$-algebras with a continuous rank function (in particular, those that have strict comparison, almost divisibility and stable rank one), the assumption of compactness of the tracial state space can be removed.

\sub{Acknowledgements} I am grateful to Wilhelm Winter, Karen Strung, Aaron Tikuisis and Stuart White for many helpful discussions.

\section{Preliminaries} \label{prelim}

In this section, we record the background results and terminology needed for the rest of the article.

\sub{Cones and traces} If $C$ is a (pointed) convex cone in a locally convex Hausdorff space, a \emph{base} of $C$ is a convex subset $X$ of $C$ such that $y\in C$ if and only if $y=\alpha x$ for unique $\alpha\ge 0$ and $x\in X$. It is straightforward to show that, if $X$ is a \emph{compact} base of $C$, then the projection $C\backslash \{0\} \to X$ is continuous, and (at least if one can choose seminorms defining the ambient topology to all be strictly positive on $C\backslash\{0\}$) every continuous affine functional on $X$ extends uniquely to a continuous linear functional on $C$. Finally, a compact base $X$ is a \emph{Choquet simplex} precisely when $C$ is a vector lattice. (See \cite{Alfsen:1971hl} or \cite{Phelps:2001rz} for details about cones and simplices.)

Let $A$ be a $\cs$-algebra. We denote by $\ped(A)$ the minimal dense ideal of $A$ (see \cite[Chapter 5]{Pedersen:1979rt}), by $T(A)$ the cone of densely finite lower semicontinuous traces on $A$ (regarded as a subset of the dual of $\ped(A)$ and equipped with the corresponding weak$^*$-topology), and by $T_1(A)$ the space of tracial states on $A$, i.e.\ those elements of $T(A)$ of norm $1$. Note that, if $A$ is simple, then every nonzero element of $T(A)$ is faithful, so in this case any compact base of $T(A)$ has the extension property mentioned in the previous paragraph. In particular, this is the case for the following (which are compact bases of $T(A)$ under the stated hypotheses):
\begin{itemize}
\item $T_1(A)$, whenever it is compact and $A$ has no unbounded traces, and
\item $T_{a\mapsto1}(A):=\{\tau\in T(A) \mid \tau(a)=1\}$, whenever $a\in\ped(A)^+$ is full (see \cite[Proposition 3.4]{Tikuisis:2012yq}).
\end{itemize}
Since $T(A)$ is a lattice (see \cite[Corollary 3.3]{Pedersen:1966zr} and \cite[Theorem 3.1]{Pedersen:1969kq}, also \cite[Theorem 3.3]{Elliott:2009kq}), any compact base of $T(A)$ is also a Choquet simplex.

For separable, simple $A$, we define
\begin{align*}
\aff(T(A))&:=\{f:T(A)\to[0,\infty) \mid \textrm{$f$ linear, continuous, } f(\tau)>0 \:\textrm{for}\: \tau\ne 0\}\\
\saff(T(A))&:=\{f:T(A)\to[0,\infty] \mid f_n\uparrow f\:\textrm{pointwise, }\: f_n\in \aff(T(A))\}.
\end{align*}
Note that the norm map $\|\cdot\|:T(A)\to[0,\infty]$, which we will sometimes denote by $\omega_A$, is an element of $\saff(T(A))$. Moreover, if $A$ has no unbounded traces, then $T_1(A)$ is compact precisely when $\omega_A$ is continuous.

\sub{The Cuntz semigroup}
The Cuntz relation on $M_\infty(A)^+:=\bigcup_{k=1}^\infty M_k(A)^+$ is defined by $a\precsim b$ if $\|a-v_nbv_n^*\| \to 0$ for some $v_n$, and $\sim$ denotes the equivalence relation that symmetrizes it. The Cuntz semigroup $W(A)$ is the set of equivalence classes $[a]$ of elements $a\in M_\infty(A)^+$. It is a positively ordered abelian monoid under the addition $[a]+[b]=[\diag(a,b)]$ and the partial order $[a]\le[b]$ if $a\precsim b$. We refer the reader to \cite{Ara:2009cs} for more information.

Every $\tau\in T(A)$ gives rise to a lower semicontinuous functional $d_\tau:W(A)\to[0,\infty]$, defined by $d_\tau([a])=\lim_{n\to\infty}\tau(a^{1/n})$. (Lower semicontinuity of $d_\tau$ means precisely that $d_\tau([a])=\sup_{\varepsilon>0}d_\tau([(a-\varepsilon)_+])$ for every $a\in M_\infty(A)^+$, where $(a-\varepsilon)_+$ is defined to be the element of $\cs(a)$ (in fact, of $\ped(\cs(a))$) corresponding under functional calculus to $f(t):=\max\{0,t-\varepsilon\}$.)

As in \cite{Dadarlat:2010nr} (see also \cite[Remark 6.0.4]{Robert:2010qy} and \cite[Corollary I.1.4	]{Alfsen:1971hl}), when $A$ is simple and separable we define
\[
\iota: W(A) \to \saff(T(A)) \quad \textrm{by} \quad \iota([a])(\tau):=d_\tau(a) \:\textrm{for}\: [a]\in W(A), \tau\in T(A).
\]
(Note that if $e\in A$ is strictly positive, then $\iota([e])=\omega_A$, i.e.\ $d_\tau(e)=\|\tau\|$ for every $\tau\in T(A)$.)

An exact $\cs$-algebra $A$ is said to have \emph{strict comparison} if, whenever $a,b\in M_\infty(A)^+$ with $d_\tau(a)<d_\tau(b)$ for every $\tau\in T(A)$ such that $d_\tau(b)=1$, then $a\precsim b$. We say that $A$ (or $W(A)$) has \emph{almost divisibility} if for every $y\in W(A)$ and $m\in\nn$ there exists $x\in W(A)$ such that $mx\le y \le (m+1)x$. By results of R{\o}rdam \cite{Rordam:2004kq}, every separable (exact) $\mathcal{Z}$-stable $\cs$-algebra has both of these properties (see \cite[Proposition 3.7]{Winter:2012pi}, the proof of which works equally well for nonunital $\cs$-algebras, and also \cite[Proposition 6.2]{Elliott:2009kq}).

\sub{Tracially large order zero maps}
Recall that a completely positive (c.p.) map has \emph{order zero} if it preserves orthogonality (see \cite{Winter:2009sf}). Given a $\cs$-algebra $A$, we denote by $A_\infty$ the central sequence algebra $A^\infty\cap A':=\ell^\infty(A)/c_0(A)\cap A'$, where $A$ is embedded as the set of constant sequences. Since order zero maps on $M_k:=M_k(\cc), k\in\nn$, are projective (by results of Loring \cite{Loring:1997it}; see also \cite[Lemma 2.1]{Toms:2012qy}), every completely positive and contractive (c.p.c.)  order zero map $\Phi:M_k \to A^\infty$ lifts to a c.p.c.\ order zero map $(\varphi_n)_{n=1}^\infty:M_k\to \ell^\infty(A)$, where each $\varphi_n$ is a c.p.c.\ order zero map $M_k\to A$. As in \cite{Toms:2012qy}, if $A$ is separable and $T_1(A)$ is a nonempty base of $T(A)$, such a $\Phi$ is said to be \emph{uniformly tracially large} if
\[
\lim_{n\to\infty} \inf_{\tau\in T_1(A)} \tau(\varphi_n(1_k))=1
\]
for some (hence every) such lifting $(\varphi_n)_{n=1}^\infty$.

For separable, simple, unital, infinite-dimensional, nuclear $\cs$-algebras with finitely many extremal tracial states, the existence of uniformly tracially large order zero maps is guaranteed by \cite[Lemma 3.3]{Matui:2012qv}. In general, one has the following.

\begin{thm}[Matui--Sato; Nawata] \label{msn}
Let $A$ be a separable, simple, nonelementary, nuclear $\cs$-algebra with strict comparison, such that every trace on $A$ is bounded and $T_1(A)$ is nonempty and compact. Suppose that for every $k\in\nn$, there is a uniformly tracially large c.p.c.\ order zero map $\Phi: M_k\to A_\infty$. Then $A$ is $\mathcal{Z}$-stable.
\end{thm}

In the unital case, strict comparison is used in \cite{Matui:2012qv} to prove that every c.p.\ map $A\to A$ can be excised in small central sequences, and hence that $A$ has property (SI) (see \cite[Definition 2.1]{Matui:2012qv} and \cite[Definition 4.1]{Matui:2012qv} for the relevant definitions). As presented in \cite{Matui:2012qv}, $\Phi$ plays a role in this part of the argument; namely, to ensure that the set of excisable c.p.\ maps is closed under finite sums. However, it is demonstrated in \cite{Kirchberg:2012uq} that this is unnecessary: compare \cite[Proposition 5.9]{Kirchberg:2012uq} with \cite[Lemma 3.1]{Matui:2012qv} (this observation was made by the referee of the published version of this article). Hence, $\Phi$ is only needed to show that $\mathcal{Z}$-stability can be deduced from property (SI): when (SI) is applied to certain central sequences obtained from $\Phi$, one exhibits in $A_\infty$ the order zero relations for the dimension drop algebra $Z_{k,k+1}$ described in \cite[Proposition 5.1]{Rordam:2009qy}; by \cite[Proposition 2.2]{Toms:2005kq}, this implies $\mathcal{Z}$-stability.

If $A$ does not necessarily have a unit, then it is shown in \cite{Nawata:2012fk} that the appropriate nonunital version of property (SI) holds (see \cite[Section 5]{Nawata:2012fk}). The $Z_{k,k+1}$-relations are then witnessed not in $A_\infty$ but in the quotient $F(A)$ of $A_\infty$ by the left annihilator of $A$, first defined in \cite{Kirchberg:2006fk}; from this, $\mathcal{Z}$-stability follows (see \cite[Proposition 4.11]{Kirchberg:2006fk} and also \cite[Proposition 5.1]{Nawata:2012fk}).

\sub{Algebraically simple $\cs$-algebras} A simple $\cs$-algebra is algebraically simple precisely when $A=\ped(A)$ (in particular, if $A$ is unital). It follows from \cite[Proposition 5.6.2]{Pedersen:1979rt}, \cite[Lemma 5.6]{Cuntz:1979fv} and Brown's theorem \cite{Brown:1977kq} that every simple, $\sigma$-unital $\cs$-algebra is stably isomorphic to an algebraically simple $\cs$-algebra. Such algebras are tracially well behaved in the following sense.

\begin{prop} \label{algsim}
The following hold for a separable algebraically simple $\cs$-algebra $A$.
\begin{enumerate}[(i)]
\item \label{a} Every trace on $A$ is bounded (so $T_1(A)$ is a base of $T(A)$).
\item \label{b} The weak$^*$-closure of $T_1(A)$ in $A^*$ does not contain $0$; equivalently,
\[
\inf_{\tau\in T_1(A)}\tau(a)>0 \quad \textrm{for every }\: a\in A^+\backslash\{0\}.
\]
\item \label{c} Suppose that $T_1(A)$ is nonempty and compact. For any strictly positive continuous affine functional $f$ on $T_1(A)$ and for any $\varepsilon>0$, there exists $a\in A^+$ with $\tau(a)=f(\tau)$ for every $\tau\in T_1(A)$ and $\|a\|<\|f\|+\varepsilon$.
\item \label{d} Let $A$ be as in (iii). Then
\[
\lim_{n\to\infty} \sup_{\tau\in T_1(A)} |\tau(e_na)-\tau(e_n)\tau(a)| = 0
\]
for any central sequence $(e_n)_{n=1}^\infty\in A_\infty$ and any $a\in A$.
\item \label{e} Suppose that $A$ is as in (iii) and is also nuclear. Then for every closed subset $X\in \partial_e(T_1(A))$ and mutually orthogonal positive functions $f_1,\ldots,f_N\in C(X)$ of norm $\le 1$, there exist central sequences $(a_{i,n})_{n=1}^\infty$ of positive contractions in $A$ such that
\[
\lim_{n\to\infty} \sup_{\tau\in X} |\tau(a_{i,n})-f_i(\tau)|=0 \quad\textrm{and}\quad \lim_{n\to\infty}\|a_{i,n}a_{j,n}\|=0 \: \textrm{ for }\: i\ne j.
\]
\end{enumerate}

\end{prop}

\begin{proof}
See \cite[Section 2]{Tikuisis:2012kx} for the proofs of the first two assertions. For the remaining three, the assumption that $T_1(A)$ is compact ensures that every continuous affine functional on $T_1(A)$ extends to a continuous linear functional on $T(A)$. (Moreover, by \cite[Theorem II.3.12]{Alfsen:1971hl}, if $X\subset \partial_e(T_1(A))$ is closed, then every continuous function on $X$ extends to a continuous affine functional on $T_1(A)$ of the same norm.) Then, (\ref{c}) is proved exactly as in \cite[Theorem 9.3]{Lin:2007qf}; Lin's theorem is stated for simple, unital $\cs$-algebras, but this assumption is only needed to invoke \cite[Corollary 6.4]{Cuntz:1979fv}, which holds for algebraically simple $\cs$-algebras in general. The final two statements are the same as \cite[Lemma 4.2 (i)]{Sato:2012fj} and \cite[Lemma 4.2 (ii)]{Sato:2012fj} respectively, and are proved in exactly the same way. Note in particular that only when appealing to \cite[Corollary 3.3]{Sato:2012fj} or \cite[Proposition 4.1]{Sato:2012fj} is a unit being used. However, the proof of the former works just as well for nonunital algebras (with the unitaries taken in the unitization), and the latter holds for algebraically simple algebras in light of (\ref{c}). 
\end{proof}

\begin{remark}
It can be shown from the results of \cite{Blackadar:1980zr} (see also \cite[Proposition 5.3]{Jacelon:2010fj}) that for every infinite-dimensional metrizable Choquet simplex $X$ with compact extreme boundary $\partial_e(X)$, there exists an algebraically simple AF algebra $A$ such that $T(A)$ has a base affinely homeomorphic to $X$, yet $T_1(A)$ is not compact.\end{remark}

\section{The main theorem} \label{mainsec}

\begin{thm} \label{main}
Let $A$ be a separable, algebraically simple, nonelementary, nuclear $\cs$-algebra and suppose that $T_1(A)$ and $\partial_e(T_1(A))$ are nonempty and compact, and $\partial_e(T_1(A))$ has finite covering dimension. Then $A$ admits a uniformly tracially large order zero map $\Phi:M_k\to A_\infty$ for every $k\in\nn$. In particular, such $\cs$-algebras are $\mathcal{Z}$-stable whenever they have strict comparison.
\end{thm}

We will use the following method of building up large order zero maps from $M_k$. (See \cite[Lemma 7.6]{Kirchberg:2012uq}, also \cite[Corollary 2.3]{Sato:2012fj}.)

\begin{prop} \label{criterion}
Let $B$ be a $\cs$-algebra, $k\ge 2$, and $\varphi_1,\ldots,\varphi_N:M_k\to B$ c.p.c.\ order zero maps whose images commute, and such that $\|\varphi_1(1_k)+\cdots+\varphi_N(1_k)\|\le 1$. Then there exists a c.p.c.\ order zero map $\psi:M_k\to B$ with $\psi(1_k)=\varphi_1(1_k)+\cdots+\varphi_N(1_k)$.
\end{prop}  

It will moreover be sufficient to verify the conditions of the proposition approximately. That is, if $A$ is separable and there exists $N\in\nn$ such that, for every finite subset $F\subset A$ and every tolerance $\varepsilon>0$, there are c.p.\ maps $\varphi_{l,n}$, $l=1,\ldots,N$, $n\in\nn$, such that $(\varphi_{l,n})_{n=1}^\infty:M_k\to A^\infty$ are c.p.\ order zero maps with commuting images and satisfy
\begin{align} \label{diagonal}
&\limsup_{n\to\infty}\left\|\sum_{l=1}^{N}\varphi_{l,n}(1_k)\right\| \le 1,\\ \nonumber
&\liminf_{n\to\infty}\inf_{\tau\in T_1(A)}\sum_{l=1}^{N} \tau(\varphi_{l,n}(1_k)) \ge 1-\varepsilon \quad\textrm{and}\\ \nonumber
&\limsup_{n\to\infty}\|[\varphi_{l,n}(x),a]\| \le \varepsilon\|x\| \:\textrm{ for }\: x\in M_k,a\in F,l=1,\ldots,N,
\end{align}
then, by a diagonal argument and Proposition~\ref{criterion}, there exists a uniformly tracially large order zero map $\psi:M_k\to A_\infty$.

\begin{proof}[Proof of Theorem~\ref{main}] The proof is exactly the same as in \cite[Lemma 3.5]{Toms:2012qy} (for the zero dimensional case) and \cite[Proposition 5.1]{Sato:2012fj} (for the general case), so we will try to just give a brief outline. Most of the work needed to carry these arguments over to nonunital $\cs$-algebras is either contained in Proposition~\ref{algsim} or has already been done by Nawata in \cite{Nawata:2012fk}. In particular, we will appeal when necessary to \cite[Proposition 5.3]{Nawata:2012fk}, which says the following: if $A$ is a $\cs$-algebra with a countable approximate unit $(h_n)_{n=1}^\infty$ and with $T_1(A)$ nonempty and compact, then
\begin{equation} \label{naw}
\lim_{n\to\infty} \sup_{\tau\in T_1(A)} |\tau(h_ng_n)-\tau(g_n)|=0
\end{equation} 
for any sequence $(g_n)_{n=1}^\infty$ of positive contractions in (the unitization of) $A$.

Finally, by $\dim$ we mean the topological covering dimension, which for separable metric spaces coincides with the (small) inductive dimension (see \cite{Engelking:1978uq}).

Fix a finite set $F\subset A$ and a tolerance $\varepsilon>0$, and suppose that $X\subset \partial_e(T_1(A))$ is closed. By the argument of \cite[Lemma 3.3]{Matui:2012qv} (see also \cite[Lemma 5.9]{Nawata:2012fk} and \cite[Lemma 2.9]{Toms:2012qy}), for every $\tau\in\partial_e(T_1(A))$ there exists a c.p.c.\ order zero map $\varphi_\tau :M_k\to A$ such that $\|[\varphi_\tau(x),a]\|<\varepsilon\|x\|$ for every $x\in M_k, a\in F$ and $\tau(\varphi_\tau(1_k))>1-\varepsilon$. The latter inequality holds on an open neighbourhood $U_\tau$ of $\tau$ in $\partial_e(T_1(A))$, and by compactness there exist $\tau_1,\ldots,\tau_M\in\partial_e(T_1(A))$ such that $X$ is covered by $\mathcal{U}:=\{U_{\tau_1},\ldots,U_{\tau_M}\}$. 

\sub{The zero-dimensional case} Suppose that $\dim X=0$. Then there exists a finite cover $\mathcal{V}=\{V_1,\ldots,V_N\}$ that refines $\mathcal{U}$ and whose elements are pairwise disjoint closed sets. For each $j\in\{1,\ldots,N\}$, let $f_j: X \to [0,1]$ be continuous with $f_j|_{V_j}=1$ and $f_j|_{\bigcup_{i\ne j}V_i}=0$. By Proposition~\ref{algsim}(\ref{e}), there exist central sequences $(a_{i,n})_{n=1}^\infty$ of positive contractions of $A$, $i=1,\ldots,N$, with
\begin{equation} \label{rep}
\lim_{n\to\infty} \sup_{\tau\in X} |\tau(a_{i,n})-f_i(\tau)|=0 \quad\textrm{and}\quad \lim_{n\to\infty}\|a_{i,n}a_{j,n}\|=0 \: \textrm{for } i\ne j.
\end{equation}
For each $j$, let $l(j)$ be such that $V_j\subset U_{\tau_{l(j)}}$. Define $\psi_n:M_k\to A$ by $\psi_n(x):=\sum_{j=1}^N a_{j,n}^{1/2}\varphi_{\tau_{l(j)}}(x) a_{j,n}^{1/2}$. Then $\psi:=(\psi_n)_{n=1}^\infty:M_k\to A^\infty$ is c.p.c.\ order zero (since each $\varphi_{\tau_{l(j)}}$ is, and the $(a_{j,n})_{n=1}^\infty$ are pairwise orthogonal). Moreover, for $x\in M_k$ and $a\in F$ we have
\[
\limsup_{n\to\infty} \|[\psi_n(x),a]\| \le \max_{j\in 1,\ldots,N} \|[\varphi_{\tau_{l(j)}}(x),a]\| < \varepsilon\|x\|,
\]
(once again using pairwise orthogonality of the $(a_{j,n})_{n=1}^\infty$), and for each $j$ and $\rho\in V_j$, we have $\rho(\varphi_{\tau_{l(j)}}(1_k))>1-\varepsilon$ (since $V_j\subset U_{\tau_{l(j)}}$). From this latter inequality, together with Proposition~\ref{algsim}(\ref{d}) and the fact that, by (\ref{rep}), $\lim_{n}\inf_{\rho\in V_j}\rho(a_{j,n})=\inf_{\rho\in V_j}f_j(\rho)=1$, it follows that
\[
\liminf_{n\to\infty} \inf_{\rho\in X}\rho(\psi_n(1_k))\ge1-\varepsilon.
\]
If $X=\partial_e(T_1(A))$, then by the Krein--Milman theorem, the conditions of (\ref{diagonal}) are satisfied.

\sub{The general case} The argument is by induction on the dimension of closed subsets of $\partial_e(T_1(A))$. Let $c$ be an integer with $1\le c\le d:=\dim\partial_e(T_1(A))$ and let $X\subset \partial_e(T_1(A))$ be closed with $\dim(X)=c$. Taking a suitable refinement $\mathcal{V}=\{V_1,\ldots,V_N\}$ of $\mathcal{U}$, and arguing as in \cite{Sato:2012fj}, we obtain the following data:
\begin{items}
\item a closed subset $X_0$ of $X$ of dimension $\le c-1$ (this uses the notion of inductive dimension; $X_0$ is the union of boundaries of elements of $\mathcal{V}$);
\item inductively, c.p.c.\ order zero maps $\psi_{l,n}:M_k\to A$, $l=0,\ldots,c-1,n\in\nn$ such that
\begin{align} \label{induction}
&\limsup_{n\to\infty}\left\|\sum_{l=0}^{c-1}\psi_{l,n}(1_k)\right\| \le 1,
&&\lim_{n\to\infty}\|[\psi_{l,n}(x),a]\| =0\:\textrm{ for }\: a\in A, x\in M_k,\\
&\lim_{n\to\infty}\inf_{\tau\in X_0}\sum_{l=0}^{c-1} \tau(\psi_{l,n}(1_k)) =1,
&&\lim_{n\to\infty}\|[\psi_{l,n}(x),\psi_{m,n}(y)]\| =0 \:\textrm{ for }\: x,y\in M_k, l\ne m; \nonumber
\end{align}
\item for each $n\in\nn$, a relatively open subset $W_{0,n}\subset X$ containing $X_0$ such that
\begin{equation} \label{wopen}
\inf_{\tau\in W_{0,n}}\sum_{l=0}^{c-1} \tau(\psi_{l,n}(1_k)) >1-\varepsilon_n, \quad \varepsilon_n \to 0;
\end{equation}
\item pairwise disjoint relatively open subsets $W_1,\ldots,W_N\subset X$ (obtained in a straightforward manner as a refinement of $\mathcal{V}$) such that $\overline{W_j} \subset U_{\tau_{l(j)}}$ for every $j$  and some $l(j)$, and such that $\{W_1,\ldots,W_N\}$ covers $X\backslash X_0$ (in particular, $\mathcal{W}_n:=\{W_{0,n},W_1,\ldots,W_N\}$ covers $X$ for every $n\in\nn$).
\end{items}
Take a partition of unity $\{f_{0,n}\}_{i=0}^N\subset C(\partial_e(T_1(A)))$ subordinate to $\mathcal{W}_n$, and use Proposition~\ref{algsim}(\ref{e}) to obtain central sequences $(a_{i,n,m})_{m=1}^\infty$ of positive contractions in $A$, $i=1,\ldots,N$, such that
\begin{equation} \label{rep2}
\lim_{m\to\infty} \sup_{\tau\in X} |\tau(a_{i,n,m})-f_{i,n}(\tau)|=0 \:\textrm{ and}\: \lim_{m\to\infty}\|a_{i,n,m}a_{j,n,m}\|=0, \: i\ne j.
\end{equation}
Fix an approximate unit $(h_m)_{m=1}^\infty$ for $A$. For every $n\in\nn$, choose a strictly increasing sequence $(k_{n,m})_{m=1}^\infty$ of positive integers such that (taking the supremum over $p\ge k_{n,m}$, $1\le j\le N$, $0\le l\le c-1$, $1\le m'\le m$ and $x$ in the unit ball of $M_k$):
\begin{equation} \label{sup}
\sup_{p,j,l,m',x} \max\{\|h_p^{1/2}a_{j,n,m'}-a_{j,n,m'}\|,  \|h_p^{1/2}\psi_{l,n}(x)-\psi_{l,n}(x)\|\} <\frac{1}{m}.
\end{equation}
Set $e_{n,m}:=h_{k_{n,m}}$, so that $(e_{n,m})_{m=1}^\infty$ is an approximate unit for $A$, and in particular represents an element of $A_\infty$. For each $n\in\nn$, since the $(a_{i,n,m})_{m=1}^\infty$ are pairwise orthogonal, we may choose positive contractions $a_{0,n,m}$, $m\in\nn$, in $A$ such that
\[
(a_{0,n,m})_{m=1}^\infty = \left(e_{n,m}^{1/2}\left(1-\sum_{i=1}^N a_{i,n,m}\right)e_{n,m}^{1/2}\right)_{m=1}^\infty \in A_\infty.
\]
Then, by (\ref{sup}), $(a_{0,n,m})_{m=1}^\infty$ commutes with $(a_{i,n,m})_{m=1}^\infty$ in $A_\infty $ for $i=1,\ldots,N$. Since $f_{0,n}=1-\sum_{i=1}^N f_{i,n}$, we also have, by (\ref{naw}) and (\ref{rep2}), that
\[
\lim_{m\to\infty} \sup_{\tau\in X} |\tau(a_{0,n,m})-f_{0,n}(\tau)|=0.
\]
By a diagonal argument, using separability of $A$, (\ref{sup}) and, for the last assertion, Proposition~\ref{algsim}(\ref{d}), we can find a subsequence $(m_n)_{n=1}^\infty$ of positive integers such that:
\begin{itemize}
\item $(a_{j,n,m_n})_{n=1}^\infty$, $j=0,1,\ldots,N$, are central sequences whose sum in $A_\infty$ is a contraction;
\item for $j=1,\ldots,N$, the $(a_{j,n,m_n})_{n=1}^\infty$ are pairwise orthogonal and all commute with $(a_{0,n,m_n})_{n=1}^\infty$ in $A_\infty$;
\item $(a_{j,n,m_n})_{n=1}^\infty$ commutes with $(\psi_{l,n}(x))_{n=1}^\infty$ in $A_\infty$ for $j=0,1,\ldots,N$, $l=0,\ldots,c-1$ and $x\in M_k$;
\item $\lim_{n} \sup_{\tau\in X} |\tau(a_{j,n,m_n})-f_{j,n}(\tau)|=0$ for $j=0,1,\ldots,N$;
\item $\lim_{n} \sup_{\tau\in X} |\tau(a_{0,n,m_n}\psi_{l,n}(x))-\tau(a_{0,n,m_n})\tau(\psi_{l,n}(x))| = 0$ for $l=0,\ldots,c-1$ and $x\in M_k$.
\end{itemize}
Then, using these properties together with (\ref{induction}) and (\ref{wopen}) and arguing just as in the zero dimensional case, the maps $\varphi_{l,n}:M_k\to A$, $l=0,\ldots,c$, defined by
\begin{align*}
\varphi_{l,n}(x) &:= a_{0,n,m_n}^{1/2}\psi_{l,n}(x)a_{0,n,m_n}^{1/2} \quad \textrm{for}\quad l=0,\ldots,c-1, \quad \textrm{and}\\
\varphi_{c,n}(x) &:= \sum_{j=1}^N a_{j,n,m_n}^{1/2}\varphi_{\tau_{l(j)}}(x)a_{j,n,m_n}^{1/2},
\end{align*}
can be shown to satisfy (\ref{diagonal}) (for $X$). Thus, by induction, there exists a uniformly tracially large order zero map $M_k\to A_\infty$. The statement about $\mathcal{Z}$-stability then follows from Theorem~\ref{msn}.
\end{proof}

Note that, as defined, uniformly tracially large order zero maps cannot be used to access algebraically simple $\cs$-algebras with non-compact tracial state space.

\begin{prop} \label{easy}
Suppose that $A$ is separable and algebraically simple, and that for some $k\in\nn$, there exists a uniformly tracially large order zero map $\Phi:M_k\to A^\infty$. Then $T_1(A)$ is compact.
\end{prop}

\begin{proof}
Let $a$ be a nonzero element of $A^+$ and choose (by (\ref{b}) of Proposition~\ref{algsim}) $M>0$ such that $\omega_A(\sigma):=\|\sigma\|<M$ for every $\sigma\in T_{a\mapsto1}$. Let $(\varphi_n)_{n=1}^\infty$ be a lift of $\Phi$ to c.p.c. order zero maps $M_k\to A$, and let $(\varepsilon_n)_{n=1}^\infty$ be a decreasing sequence of positive real numbers converging to zero. For every $n\in\nn$, let $m_n\in\nn$ be large enough so that $a_n:=\varphi_{m_n}(1_k)$ satisfies
\[
|\tau(a_n)-1|<\varepsilon_n \quad \textrm{for every } \: \tau\in T_1(A),
\]
hence
\[
|\sigma(a_n)-\|\sigma\||<M\varepsilon_n \quad \textrm{for every } \: \sigma\in T_{a\mapsto1}.
\]
Thus $\omega_A$ is a uniform limit of continuous functions, so is continuous on $T_{a\mapsto1}$. Since $T_{a\mapsto1}$ is a compact base of $T(A)$, it follows that $\omega_A$ is continuous on all of $T(A)$, and hence that $T_1(A)$ is compact.
\end{proof}

\section{Continuous rank functions} \label{further}

In this section, we consider circumstances under which we may drop the assumption of compactness of the tracial state space in the statement of Theorem~\ref{main}. Let us say that a simple, separable $\cs$-algebra $A$ with $T(A)\ne\{0\}$ \emph{has a continuous rank function} if $\iota(W(A))\cap\aff(T(A))\ne \emptyset$. This condition is trivially satisfied if $A$ contains a projection in its stabilization, or if $T_1(A)$ is a compact base of $T(A)$. Conversely, the following holds.

\begin{lemma} \label{morita}
If $A$ is separable, algebraically simple and has a continuous rank function, then there exist $k\in\nn$ and a hereditary subalgebra $B$ of $M_k(A)$ such that $T_1(B)$ is compact. 
\end{lemma}

\begin{proof}
Let $b\in M_k(A)^+$ be such that $\iota([b])$ is continuous. Then the hereditary subalgebra $\overline{b(M_k(A))b}=:B$ generated by $b$ in $M_k(A)$ is also algebraically simple. Moreover, the restriction map $\rho_B:T(A)\to T(B)$ is a linear isomorphism (see for example \cite[Proposition 4.7]{Cuntz:1979fv}), so $\omega_B=\iota([b])\circ\rho_B^{-1}$ is continuous. Hence, $T_1(B)$ is compact.
\end{proof}

\begin{lemma} \label{rank}
If $A$ is separable, simple and has almost divisibility, strict comparison and stable rank one (and $T(A)\ne\{0\}$), then there exists an algebraically simple hereditary subalgebra $B$ of $A$ that has a continuous rank function.
\end{lemma}

\begin{proof}
Let $B$ be a (nonzero) hereditary subalgebra of $A$ contained in $\ped(A)$, so that $B$ is algebraically simple. Note that $B$ inherits stable rank one, strict comparison and almost divisibility (see also the proof of \cite[Corollary 8.8]{Tikuisis:2012kx}). Then, arguing exactly as in \cite[Theorem 5.22]{Ara:2009cs}, one can show that $\iota(W(B))$ contains all of $\aff(T(B))$. Note that one has to work not with $T_1(B)$ but with $T_{a\mapsto1}(B)$ for some fixed $a\in B^+\backslash\{0\}$, and make use of (\ref{b}) and a version of (\ref{c}) of Proposition~\ref{algsim}, namely:
\begin{itemize}
\item $\sup\{\|\tau\| \mid \tau \in T_{a\mapsto1}(B)\} <\infty$, and
\item every strictly positive continuous affine functional on $T_{a\mapsto1}(B)$ is represented by some positive element of $B$ (in a norm-controlled way).
\end{itemize}
Briefly, the idea is as follows.

Almost divisibility is used to show that, for any $\varepsilon>0$ and $f\in\aff(T(B))$, there exists some $b\in M_\infty(B)^+$ such that $\iota([b])$ is uniformly close to $f$ on $T_{a\mapsto1}(B)$, within $\varepsilon$ (see \cite[Lemma 5.20]{Ara:2009cs}). Then, for a given $g\in\aff(T(B))$, one can write $g$ as the pointwise supremum of an increasing sequence $(f_n)\subset \aff(T(B))$ in such a way that, taking $b_n\in M_\infty(B)^+$ with $\iota([b_n])$ approximating $f_n$, one has $d_\tau([b_n])<d_\tau([b_{n+1}])$ for every $\tau\in T_{a\mapsto1}(B)$ (see \cite[Lemma 5.21]{Ara:2009cs}). The sequence $([b_n])_{n=1}^\infty$ is increasing (by strict comparison) and  bounded (since $\iota([b_n])$ converges pointwise to $g$). Therefore, since $B$ has stable rank one, \cite[Theorem 4.4]{Brown:2008mz} (see also \cite[Theorem 5.15]{Ara:2009cs}) implies that $[b]:=\sup[b_n]$ exists in $W(B)$ and satisfies $\iota([b])=g$.
\end{proof}

(Note that it is important for the continuous rank function to be induced by an element $b$ of some $M_k(B)$, rather than $B\otimes \mathcal{K}$, so that $b$ generates an \emph{algebraically} simple hereditary subalgebra. That is why we work with $W(B)$ rather than $\cu(B)\cong W(B\otimes\mathcal{K})$, even though (see \cite{Coward:2008rz}) suprema of increasing sequences exist in $\cu$ without the assumption of stable rank one.)

\begin{corollary} \label{cor}
Let $A$ be a simple, separable, nuclear, nonelementary $\cs$-algebra with strict comparison, whose cone of densely finite traces has as a base a Choquet simplex with compact, finite dimensional extreme boundary. Suppose also that $A$ has an algebraically simple hereditary subalgebra with a continuous rank function (for example, if $A$ has almost divisibility and stable rank one). Then $A$ is $\mathcal{Z}$-stable.
\end{corollary}

\begin{proof}
By Lemma~\ref{morita} and Brown's theorem \cite{Brown:1977kq}, there is an algebraically simple $\cs$-algebra $B$ that is stably isomorphic to $A$ and satisfies the hypotheses of Theorem~\ref{main}. The result then follows because $\mathcal{Z}$-stability is preserved under stable isomorphism \cite{Toms:2007uq}.
\end{proof}

\begin{remark}
Simple, nonelementary, approximately subhomogeneous $\cs$-algebras with no dimension growth have strict comparison and almost divisible Cuntz semigroup (see \cite[Corollary 5.9]{Tikuisis:2012ud} and \cite[Corollary 7.2]{Tikuisis:2012yq}), and also have stable rank one (see \cite[Theorem 3.2]{Santiago:2012fk}). Therefore, such algebras are $\mathcal{Z}$-stable whenever the assumption on the tracial cone also holds. This has already been shown in \cite[Corollary 9.2]{Tikuisis:2012kx}, where no such assumption is necessary, and where `no dimension growth' can be replaced by `slow dimension growth'. 
\end{remark}

\end{document}